\input ./style/arxiv-vmsta.cfg
\documentclass[numbers,compress,v1.0.1]{vmsta}

\usepackage{vtexurl}

\volume{2}
\issue{4}
\pubyear{2015}
\firstpage{371}
\lastpage{389}
\doi{10.15559/15-VMSTA44}

\startlocaldefs

\allowdisplaybreaks
\endlocaldefs

\newtheorem{theorem}{Theorem}[section]
\newtheorem{lemma}{Lemma}[section]
\newtheorem{corollary}{Corollary}[section]

\theoremstyle{definition}
\newtheorem{definition}{Definition}[section]
\newtheorem{remark}{Remark}[section]

\theoremstyle{remark}
\newtheorem{case}{Case}

\begin{document}
\begin{frontmatter}

\title{On packing dimension preservation by distribution functions of
random variables with independent $\tilde{Q}$-digits}

\author{\inits{O.}\fnm{Oleksandr}\snm{Slutskyi}}\email{slualexvas@gmail.com}
\address{Department of Mathematical Analysis and Differential Equations
of Dragomanov National Pedagogical University of Ukraine, Ukraine}

\markboth{O. Slutskyi}{PDP-transformations}

\begin{abstract}
The article is devoted to finding conditions for the packing dimension
preservation by distribution functions of random variables with
independent $\tilde{Q}$-digits.

The notion of ``faithfulness of fine packing systems for packing
dimension calculation'' is introduced, and connections between this
notion and packing dimension preservation are found.
\end{abstract}


\received{16 October 2015}
\revised{13 December 2015}
\accepted{13 December 2015}
\publishedonline{23 December 2015}
\end{frontmatter}

\section{Introduction}
Let $(M,\rho)$ be a metric space. Suppose that the
Hausdorff--Besicovitch dimension $\dim_H$ \cite
{falconer_fractal_geometry} is well defined in $(M,\rho)$. A
transformation $f: M\rightarrow M$ is called dimen\-sion-preserving
transformation \cite{torbin-s-adic-dp-2007} or $\mathit{DP}$-transformation if
\[
\quad\dim_H \bigl(f(E) \bigr)=\dim_H(E),\quad \forall E
\subset M.
\]

Let $G(M,\dim_H)$ be the set of all $\mathit{DP}$-transformations defined on
$(M,\rho)$. It is easy to see that $G$ forms a group w.r.t.\ the
composition of transformations. It is well known that any bi-Lipschitz
transformation belongs to this~group~\cite{falconer_fractal_geometry}.
However, $G$ is essentially wider than the group of all bi-Lipschitz
transformations. In 2004, some sufficient conditions for belonging of
distribution functions of random variable with independent $s$-adic
digits to group $G$ was proved by G. Torbin et al. \cite{apt2004}. There exist a
lot of $\mathit{DP}$-functions that are not bi-Lipschitz.

Sufficient conditions for distribution functions of random variables
with independent $s$-adic digits to be $\mathit{DP}$ have been found by G. Torbin
\cite{torbin-s-adic-dp-2007} in 2007. These conditions were generalized
for $Q$ by G. Torbin \cite{torbin-q-symbols-dp-2007} and later for $Q^*$-
and~$\tilde{Q}$-expansions by S. Albeverio, V. Koshmanenko, M. Pratsiovytyi, and
G. Torbin \cite{apt2008, AKPT2011}.

Recently, G. Torbin and M. Ibragim proved rather general sufficient
conditions for distribution functions of random variables with
independent $\tilde{Q}$-digits to be in $\mathit{DP}$-class.
The notion of fine covering system faithfulness for $\dim_H$
calculation \cite{AILT} plays an important role in the proof of these
conditions. This notion gives us the possibility to consider coverings
by sets from some family $\varPhi$ and to be sure that a ``dimension''
calculated in such a way is equal to $\dim_H$. Faithfulness of the
family of all $s$-adic cylinders (if $s$ is fixed) have been proven by
Billingsley \cite{billingsley-2} in 1961. Faithfulness of the family of
$Q$-cylinders have been proven by M. Pratsiovytyi and A. Turbin \cite
{turbin_prats} in 1992, and faithfulness of the family of
$Q^*$-cylinders (under the condition of separation from zero of the
corresponding coefficients) have been proven by S. Albeverio and G. Torbin
\cite{at2005} in 2005. It is necessary to remark that the last result
can be easily generalized to $\tilde{Q}$-expansion under a similar condition.

In 1982, C. Tricot \cite{tricot_two_definitions} introduced the notion of
packing dimension $\dim_P$. This dimension is in some sense dual to the
Hausdorff--Besicovitch dimension: the definition of $\dim_H$ of a set
$F$ is based on $\varepsilon$-coverings of this figure, but the
definition of $\dim_P$ is based on $\varepsilon$-packings (the
countable sets of disjoint open balls $B_k(r_k,c_k), k\in\mathbb{N}$,
with radii $r_k \leqslant\varepsilon$ and centers $c_k \in F$). The
packing dimension has all ``good'' properties of a fractal dimension,
such as the countable stability. Therefore, proving or disproving
similar results for $\dim_P$ is important. For example, we consider the
group of packing-dimension-preserving transformations (or
$\mathit{PDP}$-transformations).
\begin{definition}
The transformation $f$ is said to be a $\mathit{PDP}$-transformation if
\[
\forall E\subset M,\quad\dim_P\bigl(f(E)\bigr)=
\dim_P(E).
\]
\end{definition}
There are a lot of problems with proving of many conjectures for $\dim
_P$ because work with packings is essentially more complicated than
work with coverings \cite{mattila_geometry_of_sets}.

These problems are solving bit by bit. For example, M. Das \cite
{das-billingsley-packing-dimension} has proven the Billingsley theorem
for packing dimension;
J. Li \cite{li_pdp} 
obtained some sufficient conditions for distribution functions of
random variables with independent $\tilde{Q}$-digits to be in $\mathit{PDP}$-class.
Namely, J. Li has proven the following theorem.
\begin{theorem}
Let $F_\xi$ be the distribution function of a random variable \(\xi\)
with independent $\tilde{Q}$-representation. If $\inf_{i,j}
q_{ij}=q_*>0$ and $\inf_{i,j} p_{ij}=p_*>0$, then $F_\xi$ preserves the
packing dimension if and only if
\[
\limsup_{k\to\infty} \frac{h_1+h_2+\cdots+h_k}{b_1+b_2+\cdots+b_k}=1,
\]
where $h_j=-\sum_{i=1}^{n_j} p_{ij}\ln p_{ij}$ and $b_j=-\sum_{i=1}^{n_j} p_{ij}\ln q_{ij}$.
\end{theorem}

In Remark 4.2 at the and of article \cite{li_pdp}, we read: ``The
conditions $\inf_{i,j} q_{ij}=\break q_*>0$ and $\inf_{i,j} p_{ij}=p_*>0$ play
an important role in the proof of the theorem. Open question: What can
we say about the topic if we remove these conditions?''

S. Albeverio, M. Pratsiovytyi, and G. Torbin \cite{apt2008}
removed the condition $\inf_{i,j} p_{ij}=p_*>0$ in a similar situation
for $\mathit{DP}$-transformations.

In case of packing dimension, the approach of \cite{apt2008} is
complicated because it requires appropriate results about the fine
packing system faithfulness for packing dimension calculation. Even the
definition of the fine packing system faithfulness is a problem because
centers of all balls in packings should be in the set the dimension of
which is calculated.

The aim of this paper is to propose some alternative definition of the
packing dimension, uncentered packing dimension or $\dim_{P(\mathit{unc})}$. In
the proposed definition, the condition ``the centers of balls should be
in the figure the dimension of which is calculated'' in the definition
of $\dim_P$ is replaced by ``every ball should have a nonempty
intersection with the figure.'' We prove that, in some wide class of
metric spaces (including $\mathbb{R}^n$), the value of packing
dimension with uncentered balls is matching to the value of classical
packing dimension. Introduction of the fine packing system faithfulness
notion is very simple in the case of proposed definition. It allows us
to prove faithfulness (under the condition of separation from zero of
the coefficients) of a $\tilde{Q}$-cylinder system and sufficient
conditions for the distribution function of a random variable with
independent $\tilde{Q}$-digits to be in the $\mathit{PDP}$-class. The
corresponding theorem is the main result of the paper.

\begin{theorem}
Let $\inf_{i,j} q_{ij}:=q_{\min}$. Suppose that $q_{\min}>0$. Let
\[
\begin{aligned} & T:= \biggl\{ k: k\in\mathbb{N}, p_k<\frac{q_{\min}}{2} \biggr\};
\\
& T_k:=T\cap\{1,2,\dots,k\};
\\
& B:=\limsup_{k\rightarrow\infty} \frac{\sum_{j\in T_k} \ln\frac{1}{p_j}}{k}. \end{aligned} %
\]

Let $F_\xi$ be the distribution function of a random variable \(\xi\)
with independent $\tilde{Q}$-representation. Then $F_\xi$ preserves the
packing dimension if and only if
\[
\left\{ %
\begin{aligned} &\dim_P \mu_\xi=1;
\\
&B=0. \end{aligned} %
\right.
\]
\end{theorem}

\section{Packing dimension}
Let us recall the definition of packing dimension in the form given,
for example, in \cite{falconer_fractal_geometry}.

\begin{definition}
Let $E\subset M$ and $\varepsilon>0$. A finite or countable family $\{
E_j\}$ of open balls is called an \textit{$\varepsilon$-packing} of a
set $E$ if\vadjust{\eject}

\begin{enumerate}
\item$|E_i|\leqslant\varepsilon$ for all $i$;
\item$c_i \in E,\  i\in\mathbb{N}$, where $c_i$ is the center of the
ball $E_i$;
\item$E_i \cap E_j = \varnothing$ for all $i,j$, $i \neq j$.
\end{enumerate}
\end{definition}

\begin{remark}
The empty set of balls is a packing of any set.
\end{remark}

\begin{definition}
Let $E\subset M$, $\alpha\geqslant0$, $\varepsilon>0$. Then the \textit
{$\alpha$-dimensional packing premeasure} of a bounded set $E$ is
defined by
\[
\mathcal{P}^\alpha_{\varepsilon}(E):=\sup \biggl\{ \sum
_i |E_i|^\alpha \biggr\},
\]

\noindent where the supremum is taken over all at most countable
$\varepsilon$-packings $ \{ E_j  \}$ of $E$ (if $E_j =
\varnothing$ for all $j$, then $\mathcal{P}^\alpha_{\varepsilon}(E)=0$).
\end{definition}

\begin{definition}
The \textit{$\alpha$-dimensional packing quasi-measure} of a set $E$ is
defined by

\[
\mathcal{P}^\alpha_{0}(E):=\lim_{\varepsilon\rightarrow0}
\mathcal {P}^\alpha_{\varepsilon}(E).
\]
\end{definition}

\begin{definition}
The \textit{$\alpha$-dimensional packing measure} is defined by
\[
\mathcal{P}^\alpha_{}(E):=\inf \biggl\{ \sum
_j \mathcal{P}^\alpha_{0}(E_j):
E\subset\bigcup E_j \biggr\},
\]

\noindent where the infimum is taken over all at most countable
coverings $ \{ E_j  \}$ of $E$, $E_j\subset\mathbf{M}$.
\end{definition}

\begin{definition}
The nonnegative number
\[
\dim_{P}(E): =\inf\bigl\{\alpha: \mathcal{P}^\alpha_{}(E)=0
\bigr\}
\]
is called the \textit{uncentered packing dimension} of a set $E \subset M$.
\end{definition}

\section{Uncentered packing dimension}

\begin{definition}
Let $E\subset M$ and $\varepsilon>0$. A finite or countable family $\{
E_j\}$ of open balls is called an \textit{uncentered $\varepsilon
$-packing} of a set $E$ if

\begin{enumerate}
\item$|E_i|\leqslant\varepsilon$ for all $i$;
\item$E_i \cap E\neq\varnothing$;
\item$E_i \cap E_j = \varnothing$ for all $i,j$, $i \neq j$.
\end{enumerate}
\end{definition}

\begin{remark}
The empty set of balls is an uncentered packing of any set.
\end{remark}

\begin{definition}
Let $E\subset M$, $\alpha\geqslant0$, $\varepsilon>0$. Then the \textit
{uncentered $\alpha$-dimensional packing premeasure} of a bounded set
$E$ is defined by
\[
\mathcal{P}^\alpha_{\varepsilon(\mathit{unc})}(E):=\sup \biggl\{ \sum
_i |E_i|^\alpha \biggr\},
\]

\noindent where the supremum is taken over all at most countable
uncentered $\varepsilon$-packings $ \{ E_i  \}$ of $E$.
\end{definition}

\begin{definition}
The \textit{uncentered $\alpha$-dimensional packing quasi-measure} of
a~set~$E$ is defined by

\[
\mathcal{P}^\alpha_{0(\mathit{unc})}(E):=\lim_{\varepsilon\rightarrow0}
\mathcal {P}^\alpha_{\varepsilon(\mathit{unc})}(E).
\]
\end{definition}

\begin{definition}
\textit{Uncentered $\alpha$-dimensional packing measure} is defined by
\[
\mathcal{P}^\alpha_{(\mathit{unc})}(E):=\inf \biggl\{ \sum
_j \mathcal{P}^\alpha_{0(\mathit{unc})}(E_j):
E\subset\bigcup E_j \biggr\},
\]

\noindent where the infimum is taken over all at most countable
coverings $ \{ E_j  \}$ of $E$, $E_j\subset\mathbf{M}$.
\end{definition}

\begin{remark}
If $(M,\rho)=\mathbb{R}^1$ and $\alpha=1$, then the $\alpha
$-dimensional packing measure and uncentered $\alpha$-dimensional
packing measure are the Lebesgue measure.
\end{remark}

\begin{definition}
The nonnegative number
\[
\dim_{P(\mathit{unc})}(E): =\inf\bigl\{\alpha: \mathcal{P}^\alpha_{(\mathit{unc})}(E)=0
\bigr\}.
\]
is called the \textit{uncentered packing dimension} of a set $E \subset M$.
\end{definition}

\begin{theorem}
Let $(M,\rho)$ be a metric space. Let $C \in\mathbb{N}$. If for all $
r>0$ and for any open ball $I$ with $|I|=8r$, there exist at most
$N(I)$ balls $I_i,\ i \in\{1,\dots,~N(I)\}$ such that $I_i\subset I,\
i\in\{1,\dots,N(I)$, $|I_i|=r,~ i \in\{1,\dots,N(I)\}$, and $~N(I)
\leq C $. Then
\[
\dim_{P(\mathit{unc})}(E)=\dim_P(E).
\]
\end{theorem}
\begin{proof}

\textit{Step 1.} Let us prove the inequality $\dim_{P(\mathit{unc})}(E)\geqslant
\dim_P(E)$.

By the definitions and supremum property we have
\[
\mathcal{P}^\alpha_{r(\mathit{unc})}(E)\geqslant\mathcal{P}^\alpha_r(E).
\]

By the limit property of inequalities we have
\[
\mathcal{P}^\alpha_{0(\mathit{unc})}(E)\geqslant\mathcal{P}^\alpha_0(E).
\]

Hence,
\[
\mathcal{P}^\alpha_{(\mathit{unc})}(E)\geqslant\mathcal{P}^\alpha(E).
\]
Let $\dim_{P(\mathit{unc})}(E)=\alpha_0$. By the definition of $\dim
_{P(\mathit{unc})}(E)$ we have
\[
\forall\varepsilon>0,\quad\mathcal{P}^{\alpha_0+\varepsilon}_{(\mathit{unc})}(E)=0.
\]
Therefore,
\[
\forall\varepsilon>0,\quad\mathcal{P}^{\alpha_0+\varepsilon}_0(E)=0,
\]
and, consequently,
\[
\dim_P(E)\leqslant\alpha_0.
\]
Hence, it follows that $\dim_{P(\mathit{unc})}(E)\geqslant\dim_P(E)$, which is
our claim.

\textit{Step 2.} Let us show that $\dim_{P(\mathit{unc})}(E)\leqslant\dim_P(E)$.

If $\dim_{P(\mathit{unc})}(E)=0$, then the statement is true.

Let us consider the case $\dim_{P(\mathit{unc})}(E)\neq0$. Fix $0<t<s<\dim_{P(\mathit{unc})}(E)$.

Since
$
s<\dim_{P(\mathit{unc})}(E),
$
we have
\[
\begin{aligned} &\mathcal{P}^s_{(\mathit{unc})}(E)=+\infty,
\\
&\mathcal{P}^s_{0(\mathit{unc})}(E)=+\infty. \end{aligned} %
\]
Therefore,
\[
\forall r>0,\ \mathcal{P}^s_{r(\mathit{unc})}(E)=+\infty.
\]
From this and from the supremum property, it follows that there is an
uncentered packing $V:=\{E_i\}$ of the set $E$ with
\begin{equation}
\label{eq:dimp_unc_equal_to_dimp_first} \sum_i |E_i|^s>1.
\end{equation}
Let us divide the packing $V$ into classes
\[
V_k:= \bigl\{ E_i: 2^{-k-1}
\leqslant|E_i| < 2^{-k} \bigr\}.
\]
Let $n_k$ be the number of balls $V_k$. We will show that
\[
\exists k_0: n_{k_0} \geqslant2^{k_0 t}
\bigl(1-2^{t-s}\bigr).
\]

To obtain a contradiction, suppose that
\[
n_k < 2^{kt}\bigl(1-2^{t-s}\bigr)\quad\mbox{for
all } k.
\]
Then
\[
\sum_i |E_i|^s<\sum
_k 2^{-ks}\cdot n_k<\sum
_k 2^{-ks} \cdot2^{kt}
\bigl(1-2^{t-s}\bigr)=\bigl(1-2^{t-s}\bigr) \cdot\sum
_k \bigl(2^{t-s}\bigr)^k=1,
\]
which contradicts our assumption \eqref{eq:dimp_unc_equal_to_dimp_first}.

Therefore, such $k_0$ exists. Let us consider $V_{k_0}$. We denote by
$A_1, A_2, \dots, A_{n_{k_0}}$ the balls in $V_{k_0}$, that is,
\[
V_{k_0}= \{ A_1, A_2, \dots, A_{n_{k_0}}
\}.
\]

Fix $r:=2^{-k_0-1}$. Then the radius of any $A_i$ is less than $r$.
Let $T_i$ be a point of $A_i$ such that $T_i \in A_i \cap E$. Let $V'$
be the set of balls with the centers $T_i$ and radius $r$, that is,
\[
V'=\bigl\{A'_i: A'_i=B(T_i,r)
\bigr\}.
\]
Fix
\[
V^*=\bigl\{A^*_i: A^*_i=B(T_i,4r)\bigr\}.
\]


Let us divide the set $V'$ into classes $K_1, K_2, \dots, K_l$ as follows.

\begin{enumerate}
\item Let us take a ball \mbox{$A'_{j_1}=A'_1$} and put it in $K_1$ together
with all other balls $A'_{i} \in V'$ such that $A'_{i} \cap A'_{j_1}
\neq\varnothing$.
\item Let us take an arbitrary ball $A'_{j_2} \in V'\setminus K_1$ and
put it in $K_2$ together with all other balls $A'_{i} \in V'\setminus
K_1$ such that $A'_{i} \cap A'_{j_2} \neq\varnothing$.
%
\item Let us continue this way until $V'\setminus(K_1 \cup K_2 \cup
\cdots\cup K_l) \neq\varnothing$. Since the number of elements in a
set $V'$ is a finite, we can find such a number $l$.
\end{enumerate}

Now suppose that the balls $A'_i$ and $A'_j$ intersect each other. In
other words,\break $\rho(T_i,T_j)\leqslant2r$. Therefore, $A_j \subset A^*_i$.

The radius of $A_j$ is greater than $r/2$. By the theorem condition,
there are no more than $C$ disjoint balls with radius $r/2$ in a ball
with radius $4r$

Therefore, there are no more than $C$ balls in any class $K_i$.

Moreover, in the case $i<m$, the balls $A'_{j_i}$ and $A'_{j_m}$ do not
intersect each other. Indeed,
suppose otherwise. Then $A'_{j_m}$ is in a class $K_i$ or in a class
with number less than $i$.

Hence,
\[
V''= \bigl\{ A'_{j_1},
A'_{j_2}, \dots, A'_{j_l} \bigr\}
\]
is a centered packing of a set $E$, and the $t$-volume of this packing
is less than the $t$-volume of the uncentered packing $V_{k_0}$ no more
than $C$ times.
Therefore,
\[
\sum_{V''} \big|A'_{j_i}\big|^t
\geqslant n_{k_0}\cdot\frac{2^{-k_0 t}}{C} \geqslant2^{k_0 t}
\bigl(1-2^{t-s}\bigr)\cdot\frac{2^{-k_0 t}}{C} = \frac{1-2^{t-s}}{C}.
\]
From this it follows that
\[
\mathcal{P}^t_{2^{-k_0}}(E)\geqslant\frac{1-2^{t-s}}{C}.
\]

By the inequality $2^{-k_0}<r$ we get
\[
\mathcal{P}^t_{r}(E)\geqslant\frac{1-2^{t-s}}{C}\quad
\mbox{for all } r>0.
\]

Consequently, as $r\to0$, we get the inequality
\[
\mathcal{P}^t_{0}(E)\geqslant\frac{1-2^{t-s}}{C}.
\]

Let us show that $\mathcal{P}^t(E)\geqslant\frac{1-2^{t-s}}{C}$. Recall
the definition

\[
\mathcal{P}^t(E)=\inf \biggl\{ \sum_j
\mathcal{P}^t_{0}(E_j): E\subset\bigcup
E_j \biggr\},
\]
where the infimum is taken over all at most countable coverings $E_j$
of a set $E$.

Let $\{E_j\}$ be an at most countable covering of $E$. Since $\dim
_{P(\mathit{unc})}(E)>s$, there is $j_0$ such that $\dim_{P(\mathit{unc})}(E_{j_0})>s$
(by the countable stability of the packing dimension $\dim_{P(\mathit{unc})}$).
In other words, we have
\[
\begin{aligned} &\mathcal{P}^s_{(\mathit{unc})}(E_{j_0})=+
\infty,
\\
&\mathcal{P}^s_{0(\mathit{unc})}(E_{j_0})=+\infty.
\end{aligned} %
\]
We conclude by the part of the theorem already proved for $E$ that
\[
\mathcal{P}^t_{0}(E_{j_0})\geqslant
\frac{1-2^{t-s}}{C}
\]
and
\[
\sum_j \mathcal{P}^t_{0}(E_j)
\geqslant\frac{1-2^{t-s}}{C}.
\]
But the previous inequality is true for an arbitrary covering $\{E_j\}$
of a set $E$ and for the infimum for all coverings. Therefore,
\[
\mathcal{P}^t(E)\geqslant\frac{1-2^{t-s}}{C}
\]
and
\[
\dim_P(E)\geqslant t.
\]
Since $t$-$\dim_{P(\mathit{unc})}(E)$ can be approximated by 0, we get $\dim
_P(E) \geqslant\dim_{P(\mathit{unc})}(E)$, which completes the proof.
\end{proof}

\begin{corollary}
If $M=\mathbb{R}^n$, then $\dim_{P(\mathit{unc})}(E)=\dim_P(E)$.
\end{corollary}
\begin{proof}
Let $B_{8r}$ be a ball with radius $8r$, $B_r$ be a ball with radius
$r$, and $\lambda$ be the $n$-dimensional Lebesgue measure. Then
\[
\lambda(B_{8r})=8^n\cdot\lambda(B_r).
\]

Therefore, we can put no more than $C=8^n$ disjoint balls with radii
$r$ in a~ball with radius $8r$, which completes the proof.
\end{proof}

\subsection{Packing dimension with respect to the family of sets}
Let $\varPhi$ be a family of balls in a metric space $(M, \rho)$.

\begin{definition}
Let $E\subset M$, $\alpha\geqslant0$, $\varepsilon>0$. Then the \textit
{$\alpha$-dimensional packing premeasure} of a bounded set $E$ with
respect to $\varPhi$ is defined by

\[
\mathcal{P}^\alpha_{\varepsilon}(E,\varPhi):=\sup \biggl\{ \sum
_i |E_i|^\alpha \biggr\},
\]

\noindent where the supremum is taken over all uncentered
$\varepsilon$-packings $ \{ E_i  \} \subset\varPhi$ of $E$ (if
$\{E_i\} = \varnothing$, then $\mathcal{P}^\alpha_{\varepsilon}(E,\varPhi)=0$).
\end{definition}

\begin{definition}
The \textit{$\alpha$-dimensional packing quasi-measure} of a set $E$
w.r.t.~$\varPhi$ is defined by
\[
\mathcal{P}^\alpha_{0}(E,\varPhi):=\lim_{\varepsilon\rightarrow0}
\mathcal {P}^\alpha_{\varepsilon}(E,\varPhi).
\]
\end{definition}

\begin{definition}
The \textit{$\alpha$-dimensional packing measure} w.r.t.\ $\varPhi$ is
defined by
\[
\mathcal{P}^\alpha(E,\varPhi):=\inf \biggl\{ \sum
_j \mathcal{P}^\alpha_{0}(E_j,
\varPhi): E\subset\bigcup E_j \biggr\},
\]

\noindent where the infimum is taken over all at most countable
coverings $ \{ E_j  \}$ of $E$, $E_j\subset\mathbf{M}$.\vadjust{\eject}
\end{definition}

\begin{definition}
The nonnegative number
\[
\dim_{P}(E,\varPhi): =\inf\bigl\{\alpha: \mathcal{P}^\alpha(E,
\varPhi)=0\bigr\}
\]
is called the \textit{packing dimension} of a set $E \subset M$ w.r.t.
$\varPhi$.
\end{definition}

\begin{remark}
In the definition of $\dim_P(E,\varPhi)$, we used uncentered packing. But
we will denote this dimension without index $(\mathit{unc})$ because:
\begin{enumerate}
\item We will work in $\mathbb{R}^n$. In this space, centered and
uncentered packing dimensions are equal;
\item The centered packing dimension w.r.t. some family of balls is not defined.
\end{enumerate}
\end{remark}

\begin{theorem}
\label{dimpfi_leq_dimp}
\[
\dim_{P}(E,\varPhi)\leqslant\dim_{P(\mathit{unc})}(E).
\]
\end{theorem}
\begin{proof}
Let $\varPhi_0$ be the family of all open balls of $M$. Then
\[
\mathcal{P}^\alpha_{r(\mathit{unc})}(E)=\mathcal{P}^\alpha_{r}(E,
\varPhi_0).
\]

\noindent Since $\varPhi\subseteq\varPhi_0$, by the supremum property we have
\[
\mathcal{P}^\alpha_{r}(E,\varPhi)\leqslant
\mathcal{P}^\alpha_{r}(E,\varPhi_0).
\]

\noindent By the inequality for packing premeasures it follows that
\[
\dim_{P}(E,\varPhi)\leqslant\dim_{P(\mathit{unc})}(E),
\]
which proves the theorem.
\end{proof}

\section{Faithfulness of the open balls families for packing dimension
calculation}

\begin{definition}
Suppose that some open balls family $\varPhi$ satisfies the following condition:
for all $E \subset M$, $\dim_{P(\mathit{unc})}(E,\varPhi)=
\dim_{P(\mathit{unc})}(E).
$
Then $\varPhi$ is said to be \emph{faithful for uncentered packing
dimension calculation}.
\end{definition}

\begin{remark}
The notion of faithfulness is introduced for the Hausdorff--Besicovitch
dimension $\dim_H$ \cite{nikiforov-torbin-tvims}.
It is clear that
\[
\forall\varPhi\subset2^M, \dim_H(E,\varPhi)\geqslant
\dim_H(E).
\]
\end{remark}

\begin{theorem}[The sufficient condition for the open-ball family to be
faithful for packing dimension calculation]
\label{dost-umova-dov}
Suppose that
\begin{enumerate}
\item[\rm1.] $\varPhi$ is a family of intervals from $[0;1]$;
\item[\rm2.] $\exists C>0: \forall(a;b)\subset[0;1]$, $\exists\varDelta
(a;b)\in\varPhi$ such that:
\begin{enumerate}
\item[\rm(a)] $\frac{a+b}{2}\in\varDelta(a,b)$;
\item[\rm(b)] $\varDelta(a,b)\subset(a;b)$;
\item[\rm(c)] $\frac{b-a}{|\varDelta(a,b)|}\geqslant C$.
\end{enumerate}
\end{enumerate}
Then $\varPhi$ is a faithful open-ball family for packing dimension
calculation.\vadjust{\eject}
\end{theorem}
\begin{proof}
Let $E$ be any set, $\alpha\geqslant0$, and $r>0$. Let $\{E_i\}=\{
(a_i;b_i)\}$ be a family of disjoint intervals such that $\frac
{a_i+b_i}{2}\in E$ and $b_i-a_i<r$.

Then the following inequality holds:
\[
\sum_i |E_i|^\alpha
\leqslant \sum_i \big|\varDelta(a_i,b_i)\big|^\alpha
\cdot C^\alpha.
\]

Taking the supremum (over all sets of intervals $\{E_i\}$ satisfying
the previous conditions), we have
\[
\mathcal{P}^\alpha_r(E) \leqslant \sup_{\{E_i\}}\big|
\varDelta(a_i,b_i)\big|^\alpha\cdot C^\alpha.
\]

Any set of intervals $\{\varDelta(a_i,b_i)\}$ satisfies the conditions
from the $\mathcal{P}^\alpha_{r(\mathit{unc})}(E,\varPhi)$ definition. So,
\[
\sup_{\{E_i\}}\big|\varDelta(a_i,b_i)\big|^\alpha
\cdot C^\alpha \leqslant \mathcal{P}^\alpha_{r(\mathit{unc})}(E,
\varPhi) \cdot C^\alpha.
\]
Therefore,
\[
\mathcal{P}^\alpha_r(E) \leqslant \mathcal{P}^\alpha_{r(\mathit{unc})}(E,
\varPhi) \cdot C^\alpha.
\]
Taking the limit of both sides, we have
\[
\mathcal{P}^\alpha_0(E) \leqslant \mathcal{P}^\alpha_{0(\mathit{unc})}(E,
\varPhi) \cdot C^\alpha.
\]

Taking the infimum over all possible coverings of the set $E$, we have
\[
\mathcal{P}^\alpha(E) \leqslant \mathcal{P}^\alpha_{(\mathit{unc})}(E,
\varPhi) \cdot C^\alpha
\]
and
\[
\dim_P(E)\leqslant\dim_{P(\mathit{unc})}(E,\varPhi).
\]
Since $[0;1]\subset\mathbb{R}^1$, it follows that
\[
\dim_P(E)=\dim_{P(\mathit{unc})}(E)
\]
and
\[
\dim_{P(\mathit{unc})}(E) \leqslant\dim_{P(\mathit{unc})}(E,\varPhi).
\]
Using
\[
\dim_{P(\mathit{unc})}(E) \geqslant\dim_{P(\mathit{unc})}(E,\varPhi)\quad\mbox{for all
} \varPhi,
\]
we obtain that $\varPhi$ is a faithful open-ball family for the packing
dimension calculation.
\end{proof}

\section{Sufficient conditions for $\tilde{Q}$-expansion cylindric
interval family to be faithful}

The $\tilde{Q}$-expansion of real numbers is a generalization of
$s$-expansion and $Q$-expansion and was described, for example, in \cite
{AKPT2011}.

\begin{theorem}
\label{faithfulness-of-tilde-q}
Let $\varPhi$ be the system of cylindric intervals of some $\tilde
{Q}$-expansion. Suppose that
\[
\inf_{i,j} q_{ij}=q_{\min}>0.
\]
Then $\varPhi$ is a faithful ball family for packing dimension calculation.
\end{theorem}
\begin{proof}
Let $\tilde{Q}_0$ be the set of $\tilde{Q}$-rational points, and $E'$
be any subset of $[0;1]$. Let $E=E'\setminus\tilde{Q}_0$. Since $\tilde
{Q}_0$ is countable, it follows that
\[
\dim_{P(\mathit{unc})}(\tilde{Q}_0)=0, \qquad \dim_{P(\mathit{unc})}(\tilde{Q}_0,\varPhi)=0,
\]
and
\[
\dim_{P(\mathit{unc})}\bigl(E'\bigr)=\dim_{P(\mathit{unc})}(E), \qquad
\dim_{P(\mathit{unc})}\bigl(E',\varPhi\bigr)=\dim_{P(\mathit{unc})}(E,
\varPhi).
\]
The proof is completed by showing that
\[
\dim_{P(\mathit{unc})}(E)=\dim_{P(\mathit{unc})}(E,\varPhi)
\]
for every set $E\subset[0;1]$ if $E$ does not contain $\tilde
{Q}$-rational points.

Let $(a;b)\subset[0;1]$. Let $\varDelta(a,b)$ be the $\tilde
{Q}$-cylindric interval of the minimal rank such that
\[
\frac{a+b}{2} \in \varDelta(a;b) \subset(a;b).
\]

Denote the rank of $\varDelta(a,b)$ by $k$. Since this rank is minimal, it
follows that $(a;b)$ is a subset of one or two cylinders with rank
$k-1$. Let us denote the cylinder with rank $k-1$ that contains $\varDelta
(a,b)$ by $\varDelta'$. If the second cylinder exists, then we denote it
by $\varDelta''$.

Let us consider the following two cases.

\begin{case}
The $\varDelta''$ does not exist. Then
\[
\big|\varDelta(a,b)\big|\leqslant b-a\leqslant\big|\varDelta'\big|,
\]
and, therefore,
\[
\big|\varDelta(a,b)\big|\geqslant(b-a)\cdot q_{\min}.
\]
\end{case}

\begin{case}
The $\varDelta''$ exists. Then
\[
\big|\varDelta'\big|\cdot2\geqslant b-a \quad \Rightarrow \quad \big|\varDelta(a,b)\big|
\geqslant(b-a)\cdot\frac{q_{\min}}{2}.
\]
\end{case}

\textbf{Summary of the two cases.}
For every interval $(a;b)$, there exists a $\tilde{Q}$-cylindric
interval $\varDelta(a,b)$ such that $\frac{a+b}{2}\in\varDelta(a;b)$ and
\[
\big|\varDelta(a,b)\big|\geqslant(b-a)\cdot\frac{q_{\min}}{2}.
\]
It follows that the family $\varPhi$ satisfies the conditions of Theorem
\ref{dost-umova-dov} and is faithful for packing dimension calculation.
\end{proof}

\begin{corollary}
Let $\varPhi$ be a family of $Q^*$-cylinders under the condition $\inf_{i,j} q_{ij}>0$. Then $\varPhi$ is faithful for packing dimension calculation.
\end{corollary}

\begin{corollary}
Let $\varPhi$ be a family of $Q$-cylinders. Then $\varPhi$ is faithful for
packing dimension calculation.
\end{corollary}

\begin{corollary}
Let $\varPhi$ be a family of $s$-adic cylinders. Then $\varPhi$ is faithful
for packing dimension calculation.
\end{corollary}

\section{Proof of the main result}

To prove the main result, we need the following two lemmas.

\begin{lemma}
\label{first-lemma}
Let $\tilde{Q}$ be the matrix $\|q_{ik}\|$, $i\in\mathbb{N}$, $k\in\{
0,1,\dots,N_k-1\}$. If
\[
\lim_{i\rightarrow\infty} \frac
{\ln q_{i_k k}}{
\ln(q_{i_1 1} q_{i_2 2} \dots q_{i_{k-1} (k-1)})} =0
\]
for every sequence $(i_k)$, then the open-ball family $\varPhi$ of the
respective expansion cylinder interiors is faithful for packing
dimension calculation.
\end{lemma}
\begin{proof}
Let us fix a set $E\subset[0;1]$. Let us fix any numbers $m\in\mathbb
{N}$, $\delta>0$ and consider the following sets:
\[
W_{m,\delta}= \biggl\{ x\in E: \frac
{\ln q_{i_k k}(x)}{
\ln(q_{i_1 1}(x) q_{i_2 2}(x) \dots q_{i_{k-1} (k-1)}(x))}<\delta, \ \forall k
\geqslant m \biggr\}.
\]

Fix some value $m$ and consider any set $W_{m,\delta}$ corresponding to
this value. There exists $\varepsilon>0$ such that $|c_m|\geqslant
\varepsilon$ for any cylinder $c_m$ of rank $m$. Consider the centered
$\varepsilon$-packing of the set $W_{m,\delta}$ by intervals $E_j$.

For every interval $E_j$, there exists a cylindric interval $\varDelta
(E_j)$ such that:
\begin{enumerate}
\item$\varDelta(E_j)\subset E_j$;
\item$\varDelta(E_j)$ contains the middle point $x_j$ of the $E_j$;
\item$\varDelta(E_j)$ has the minimal possible rank. We denote this rank
by $i_j$.
\end{enumerate}

We will say that the cylinder $\varDelta'(E_j)$ is the ``father'' of
$\varDelta(E_j)$ if $\varDelta'(E_j)\supset\varDelta(E_j)$ and the rank of
$\varDelta'(E_j)$ is equal to $i_j-1$. It is obvious that $|\varDelta
'(E_j)|\geqslant\frac{|E_j|}{2}$. Therefore,
\[
|E_j| \leqslant\frac{2|\varDelta(E_j)|}{q_{i_{k_j} k_j}(x_j)}, \quad \text{where } x_j
\in W_{m,\delta}.
\]

Let us estimate the $\alpha$-volume of packing of the set $E$ by
intervals $E_j$:
\[
\sum_k |E_j|^\alpha\leqslant
\sum_k \big|\varDelta(E_j)\big|^\alpha
\cdot \biggl( \frac{2}{q_{i_{k_j} k_j}(x_j)} \biggr)^\alpha.
\]

This inequality is equivalent to
\[
\sum_k |E_j|^\alpha\leqslant
\sum_k \big|\varDelta(E_j)\big|^{\alpha-\delta}
\cdot\big|\varDelta(E_j)\big|^\delta\cdot \biggl( \frac{2}{q_{i_{k_j} k_j}(x_j)}
\biggr)^\alpha.
\]
Let us estimate the expression
\begin{align*}
&\ln \biggl(\big|\varDelta(E_j)\big|^\delta\cdot \biggl(\frac{2}{q_{i_{k_j} k_j}(x_j)} \biggr)^\alpha \biggr)\\
&\quad = \delta\ln\bigl(q_{i_1 1}(x)
q_{i_2 2}(x) \dots q_{i_{k_j-1}
(k_j-1)}(x)\bigr)+\alpha\ln2-\alpha\ln
q_{i_{k_j} k_j}(x_j).
\end{align*}
Since $x_j \in W_{m,\delta}$, it follows that
\[
\delta\ln \bigl(q_{i_1 1}(x) q_{i_2 2}(x) \dots q_{i_{k_j-1}
(k_j-1)}
\bigr) \leqslant\ln q_{i_{k_j} k_j}(x_j).
\]
Therefore,
\[
\ln \biggl(\big|\varDelta(E_j)\big|^\delta\cdot \biggl(
\frac{2}{q_{i_{k_j} k_j}(x_j)} \biggr)^\alpha \biggr) \leqslant\alpha\ln2 + (1-\alpha)
\ln q_{i_{k_j} k_j}(x_j) \leqslant\alpha\ln2.
\]
Thus, we have
\[
\sum_j |E_j|^\alpha
\leqslant 2\sum_j \big|\varDelta(E_j)\big|^{\alpha-\delta}.
\]
Take the suprema over all possible centered packings $\{E_j\}$ of both
parts of the previous inequality:
\[
\mathcal{P}^{\alpha}_\varepsilon(W_{m,\delta}) \leqslant2\mathcal
{P}^{\alpha-\delta}_{\varepsilon(\mathit{unc})}(W_{m,\delta},\varPhi).
\]
Take the limit as $\varepsilon\rightarrow0$:
\[
\mathcal{P}^{\alpha}_0(W_{m,\delta}) \leqslant2
\mathcal{P}^{\alpha
-\delta}_{0(\mathit{unc})}(W_{m,\delta},\varPhi).
\]
We obtain that
\[
\mathcal{P}^{\alpha}(W_{m,\delta}) \leqslant2\mathcal{P}^{\alpha-\delta
}_{(\mathit{unc})}(W_{m,\delta},
\varPhi).
\]
Denote $\alpha_0=\dim_P(W_{m,\delta})$. Then for all $\alpha<\alpha_0$,
the left part is equal to infinity. Thus, for all $\alpha<\alpha_0$,
the right part is equal to infinity too. It follows that
\[
\dim_{P(\mathit{unc})}(W_{m,\delta},\varPhi) \geqslant\alpha-\delta
\]
and
\[
\dim_{P(\mathit{unc})}(W_{m,\delta},\varPhi) \geqslant\dim_P(W_{m,\delta})-
\delta.
\]
Using the definition of $W_{m,\delta}$, we get
\[
E=\bigcup_{m=1}^\infty W_{m,\delta}.
\]
Now, by packing dimension countable stability,
\[
\dim_{P(\mathit{unc})}(E,\varPhi) \geqslant\dim_P(E)-\delta.
\]
Since $\delta$ can be arbitrarily small,
\[
\dim_{P(\mathit{unc})}(E,\varPhi) \geqslant\dim_P(E).
\]
To complete the proof, it remains to note that $E$ is any subset of
$[0;1]$. Thus, $\varPhi$ is faithful.
\end{proof}

\begin{lemma}
Let $\varPhi$ be a family of $\tilde{Q}$-expansion cylinders under the
condition\break $\inf q_{ij}>0$. Let $F_\xi$ be a distribution function of a
random variable \(\xi\) with independent $\tilde{Q}$-digits. Assume
that the following condition holds:\vadjust{\eject}
\begin{equation}
\label{ln_ratio} \lim_{n\rightarrow\infty}\frac
{\ln\lambda(F(\varDelta_n(x)))}{
\ln\lambda(\varDelta_n(x))}=1, \quad\forall x
\in[0;1],
\end{equation}
where $\varDelta_n(x)$ is the $n$-rank cylinder that contains $x$.

Then $\varPhi'=F(\varPhi)$ is faithful for packing dimension calculation.
\end{lemma}

\begin{proof}
$\varPhi'$ is the family of cylinders for some $\tilde{Q}$-expansion.
Denote this expansion by $\tilde{Q}'$ and the corresponding numbers
$q_{ij}$ by $q'_{ij}$.

It is not clear that condition $\inf q'_{ij}>0$ holds, so we cannot
Theorem~\ref{faithfulness-of-tilde-q}.

Let us show that the conditions of Lemma~\ref{first-lemma} hold for
this expansion. We have
\[
F\bigl(\varDelta^{\tilde{Q}}_{a_1 a_2 \dots a_n}(x)\bigr)=\varDelta^{\tilde{Q}'}_{a_1
a_2 \dots a_n}(x)
\]
and
\[
\ln\lambda\bigl(F\bigl(\varDelta_n(x)\bigr)\bigr)=\ln
\bigl(q'_{1a_1} q'_{2 a_2} \dots
q'_{n a_n}\bigr).
\]
Denote
\[
M=\limsup_{i\rightarrow\infty} \frac
{\ln q'_{ij_i}}{
\ln(q'_{1j_1} q'_{2j_2} \dots q'_{(i-1)j_{i-1}})}.
\]
To estimate $M$, we need the following equation:
\[
\lim_{n\rightarrow0}\frac
{\ln\lambda(F(\varDelta_n(x)))}{
\ln\lambda(\varDelta_n(x))}= \lim_{n\rightarrow0}
\frac
{\ln(q'_{1j_1} q'_{2j_2} \dots q'_{(i-1)j_{i-1}})+\ln q'_{ij_i}}{
\ln(q_{1j_1} q_{2j_2} \dots q_{(i-1)j_{i-1}})+\ln q_{ij_i}}.
\]
Dividing the nominator and denominator of the last fraction by\break $\ln
(q_{1j_1} q_{2j_2} \dots q_{(i-1)j_{i-1}})$, we obtain
\[
\lim_{n\rightarrow0}\frac
{1+\frac{\ln q'_{ij_i}}{\ln(q'_{1j_1} q'_{2j_2} \dots q'_{(i-1)j_{i-1}})}}{
\frac{\ln(q_{1j_1} q_{2j_2} \dots q_{(i-1)j_{i-1}})}{\ln(q'_{1j_1}
q'_{2j_2} \dots q'_{(i-1)j_{i-1}})}+\frac{\ln q_{i j_i}}{\ln(q'_{1j_1}
q'_{2j_2} \dots q'_{(i-1)j_{i-1}})}}=\frac{1+M}{1+0}=1 \quad \Rightarrow \quad
M=0.
\]
It follows that $\tilde{Q}'$ satisfies the conditions of Lemma~\ref{first-lemma},
and therefore $\varPhi'$ is faithful.
\end{proof}

\begin{proof}[Proof of the main result]

Let us show that if $F_\xi$ is $\mathit{PDP}$, then $\dim_P(\mu_\xi)=1$.

Assume the converse. Then there exists a set $E_\alpha$ such that $\mu
_\xi(E_\alpha)=1$ and $\dim_P(E_\alpha)=\alpha$. Consider $F_\xi
(E_\alpha)$. Since $\mu_\xi(E_\alpha)=1$, we have $\lambda(E_\alpha
)=1$, and thus $\dim_P(F_\xi(E_\alpha))=1$.

We obtain the following inequality:
\[
\dim_P\bigl(F_\xi(E_\alpha)\bigr)=1\neq\alpha=
\dim_P(E_\alpha),
\]
and this contradicts the assumption that $F_\xi$ is $\mathit{PDP}$. Therefore,
we will show that if $F_\xi$ is $\mathit{PDP}$, then $\dim_P(\mu_\xi)=1$.

The next part of the proof consists of two steps:
\begin{enumerate}
\item If $\dim_P(\mu_\xi)=1$ and $B=0$, then $F_\xi$ is $\mathit{PDP}$;
\item If $\dim_P(\mu_\xi)=1$ and $B\neq0$, then $F_\xi$ is not
$\mathit{PDP}$.\vadjust{\eject}
\end{enumerate}

Let $\varepsilon$ be some positive number such that $\varepsilon<\frac
{1}{2}q_{\min}$. Consider the following sets:
\[
\begin{aligned} & T^+_{\varepsilon,k}= \bigl\{j: j\in\mathbb{N}, j
\leqslant k, |p_{ij}-q_{ij}|\leqslant\varepsilon, \ i\in
\{0,1,\dots,s-1\} \bigr\},
\\
& T^-_{\varepsilon,k}= \{1,2,\dots,k \}\setminus T^+_{\varepsilon,k},
\\
& T= \biggl\{k: k\in\mathbb{N}, p_k < \frac{1}{2}q_{\min}
\biggr\},
\\
& T_k=T\cap\{1,2,\dots,k\},
\\
& T_{\varepsilon,k}=T^-_{\varepsilon,k} \setminus T_k. \end{aligned}
\]

\textit{Step 1.} Let us show that if $\dim_P(\mu_\xi)=1$ and $B=0$,
then $F_\xi$ is $\mathit{PDP}$.
Since $B=0$, we see that
\[
\lim_{k\rightarrow\infty} \frac{\sum_{j\in T_k} \ln{p_j}}{k\ln{q_{\min}}}=0.
\]

Consider the fraction
\[
\begin{aligned}
&\frac{\ln\mu_\xi(\varDelta_{a_1 a_2 \dots a_k(x)})}{\ln\lambda(\varDelta_{a_1 a_2 \dots a_k(x)})}\\
&\quad  = \frac{\sum_{j\in T^+_{\varepsilon,k}} \ln{p_{a_j(x)j}}+\sum_{j\in
T_{\varepsilon,k}} \ln{p_{a_j(x)j}}+\sum_{j\in T_k} \ln
{p_{a_j(x)j}}}{\sum_j \ln{q_{a_j(x) j}}}. \end{aligned} %
\]

Split this fraction into three terms. Consider the first term
\[
\frac{\sum_{j\in T^+_{\varepsilon,k}} \ln{p_{a_j(x)j}}}{\sum_j \ln
{q_{a_j(x)j}}}.
\]

It is easy to prove that
\[
\begin{aligned}
 \sum_{j\in T^+_{\varepsilon,k}}\ln{p_{a_j(x)j}} &\geqslant \sum_{j\in T^+_{\varepsilon,k}}
\ln(q_{a_j(x)j}-\varepsilon)\\
&=  \sum_{j\in T^+_{\varepsilon,k}} \biggl( \ln{q_{a_j(x)j}}+\ln
\biggl(\frac{q_{a_j(x)j}-\varepsilon}{q_{a_j(x)j}} \biggr) \biggr) \\
&\geqslant \sum_{j\in T^+_{\varepsilon,k}} ( \ln{q_{a_j(x)j}} )+|T^+_{\varepsilon,k}|\cdot
\frac{2\varepsilon}{q_{\min}}, \end{aligned} %
\]
where $|T^+_{\varepsilon,k}|$ is the number of elements in
$T^+_{\varepsilon,k}$. On the other hand,
\[
\sum_{j\in T^+_{\varepsilon,k}} \ln{p_{a_j(x)j}} \leqslant \sum
_{j\in T^+_{\varepsilon,k}} ( \ln{q_{a_j(x)j}} )-|T^+_{\varepsilon,k}|
\cdot\frac{2\varepsilon}{q_{\min}}.
\]
Also,
\[
1+\frac{|T^+_{\varepsilon,k}|\cdot2\varepsilon}{q_{\min}\cdot\sum_{j=0}^k \ln{q_{a_j(x)j}}} \leqslant\lim_{k\rightarrow\infty} \frac{\sum_{j\in T^+_{\varepsilon
,k}} \ln{p_{a_j(x)j}}}{\sum_{j=0}^k \ln{q_{a_j(x)j}}}
\leqslant1-\frac{|T^+_{\varepsilon,k}|\cdot2\varepsilon}{q_{\min
}\cdot\sum_{j=0}^k \ln{q_{a_j(x)j}}}
\]
(note that $\frac{|T^+_{\varepsilon,k}|\cdot2\varepsilon}{q_{\min
}\cdot\sum_{j=0}^k \ln{q_{a_j(x)j}}}<0$).\vadjust{\eject}

Since $q_{\min}\leqslant q_{ij}\leqslant q_{\max}$ and
$|T^+_{\varepsilon,k}|\leqslant k$, we have
\[
1+\frac{k\cdot2\varepsilon}{q_{\min}\cdot k \ln{q_{\max}}} \leqslant\lim_{k\rightarrow\infty} \frac{\sum_{j\in T^+_{\varepsilon
,k}} \ln{p_{a_j(x)j}}}{\sum_{j=0}^k \ln{q_{a_j(x)j}}}
\leqslant1-\frac{k\cdot2\varepsilon}{q_{\min}\cdot k \ln{q_{\max}}}
\]
and
\[
1+\frac{2\varepsilon}{q_{\min}\cdot\ln{q_{\max}}} \leqslant\lim_{k\rightarrow\infty} \frac{\sum_{j\in T^+_{\varepsilon
,k}} \ln{p_{a_j(x)j}}}{\sum_{j=0}^k \ln{q_{a_j(x)j}}}
\leqslant1-\frac{2\varepsilon}{q_{\min}\cdot\ln{q_{\max}}}.
\]
Since $\varepsilon$ can be arbitrarily small, it follows that
\[
\lim_{k\rightarrow\infty} \frac{\sum_{j\in T^+_{\varepsilon,k}} \ln
{p_{a_j(x)j}}}{\sum_{j=0}^k \ln{q_{a_j(x)j}}}=1.
\]

Similarly,
\[
|T_{\varepsilon,k}|\ln \biggl(\frac{q_{\min}}{2} \biggr) \leqslant\sum
_{j\in T_{\varepsilon,k}} \ln{p_{a_j(x)j}} \leqslant|T_{\varepsilon,k}|\ln
\biggl(\frac{2-q_{\min}}{2} \biggr).
\]

Therefore,
\[
\frac{\sum_{j\in T_{\varepsilon,k}} \ln{p_{a_j(x)j}}}{k \ln{q_{\min}}} \leqslant \frac{|T_{\varepsilon,k}|\ln (\frac{q_{\min}}{2} )}{k \ln
{q_{\min}}} \leqslant \frac{|T_{\varepsilon,k}|(\ln(q_{\min})+\ln(1/2))}{k \ln{q_{\min}}},
\]
and the second term tends to zero as $k\rightarrow\infty$.

Consider the third term
\[
\frac{\sum_{j\in T_k} \ln{p_{a_j(x)j}}}{\sum_j \ln{q_{a_j(x)j}}}.
\]
It can be estimated by
\[
\frac{\sum_{j\in T_k} \ln{p_j}}{k\ln{q_{\min}}},
\]
and this value tends to zero as $k\rightarrow\infty$ because $B=0$.

We obtain that
\[
\lim_{k\rightarrow\infty} \frac{\mu_\xi(\varDelta_{a_1 a_2 \dots a_{k}(x)})}{\lambda(\varDelta_{a_1 a_2
\dots a_{k}(x)})}=1.
\]

Denote by $\varPhi$ the cylinder family of given $\tilde{Q}$-expansion.
Denote the image of $\varPhi$ by $\varPhi'=F_\xi(\varPhi)$.

Using the Billingsley theorem for packing dimension \cite
{slutskyi-torbin-billingsley-theorem-analog-for-pack-dim}, we have
\[
\dim_P(E,\varPhi)=1\cdot\dim_P\bigl(F_\xi(E),
\varPhi'\bigr)\quad\forall E\subset[0;1].
\]

To prove that $\dim_P(E)=\dim_P(F_\xi(E))$, it suffices to prove that
$\varPhi$ and $\varPhi'$ are faithful.

Faithfulness of $\varPhi$ is already proved. Faithfulness of $\varPhi'$ was
proved in Lemma~1 and Lemma 2. So, we have that $\dim_P(E)=\dim_P(F_\xi
(E))$ and $F_\xi$ is a $\mathit{PDP}$-trans\-formation.\vadjust{\eject}

\textit{Step 2.}
Let us show that if $\dim_P(\mu_\xi)=1$ and $B>0$, then $F_\xi$ is not $\mathit{PDP}$.

Similarly to step 1, consider the fraction
\[
\begin{aligned} &\frac{\mu(\varDelta_{a_1 a_2 \dots a_k(x)})}{\lambda(\varDelta_{a_1 a_2
\dots a_k(x)})}\\
&\quad  = \frac{\sum_{j\in T^+_{\varepsilon,k}} \ln{p_{a_j(x)j}}+\sum_{j\in
T_{\varepsilon,k}} \ln{p_{a_j(x)j}}+\sum_{j\in T_k} \ln
{p_{a_j(x)j}}}{\sum_j \ln{q_{a_j(x)j}}} \end{aligned} %
\]
and split it into three terms. It is easy to see that the first term
tends to 1 and the second term tends to 0 (as $k\rightarrow\infty$).
Consider the third term.

Since $B>0$, there exists a subsequence $(k_m)$ such that
\[
\lim_{m\rightarrow\infty} \frac{\sum_{j\in T_{k_m}} \ln\frac{1}{p_j}}{k_m}=B.
\]
Consider the set
\[
L= \left\{ x: x=\varDelta_{a_1 a_2 \dots a_k \dots}; \Bigg\{ %
\begin{aligned} &
a_k \in\{0,1,\dots,s-1\} \text{ if } k\notin T
\\
& a_k=n_k, \text{ if } k\in T, \text{ where }
p_{n_k k}=\min_i p_{ik} \end{aligned}
\right\}.
\]
Since the digits are in infinitely many places, it follows that $\lambda
(L)=0$. But combining
\[
\lim_{m\rightarrow\infty} \frac{|T_{k_m}|}{k_m}=0
\]
and the formula for $\dim_P(\mu_\xi)$, we have $\dim_P(L)=1$.
It follows that
\[
\forall x\in L \quad\lim_{m\rightarrow\infty} \frac{\ln\mu(\varDelta_{a_1
a_2 \dots a_{k_m}(x)})}{\ln\lambda(\varDelta_{a_1 a_2 \dots a_{k_m}(x)})}=1+B.
\]
Thus, for any $\delta>0$, there exists $m(\delta)$ such that for all
$m>m(\delta)$, we have
\[
1+B-\delta\leqslant \frac{\ln\mu(\varDelta_{a_1 a_2 \dots a_{k_m}(x)})}{\ln\lambda(\varDelta_{a_1
a_2 \dots a_{k_m}(x)})} \leqslant1+B+\delta.
\]
Thus, we have
\[
\liminf_{k\rightarrow\infty} \frac
{\ln\mu(\varDelta_{a_1 a_2 \dots a_k(x)})}{
\ln\lambda(\varDelta_{a_1 a_2 \dots a_k(x)})} \geqslant1+B-\delta,
\]
and (using the Billingsley theorem for $\dim_P$)
\[
\dim_{P-\mu} (L) \cdot(1+B-\delta) \leqslant\dim_P (L),
\]
that is,
\[
\dim_P\bigl(F_\xi(L)\bigr)\leqslant\frac{1}{1+B-\delta}.
\]
Since the last inequality holds for any $\delta$, it follows that
\[
\dim_P\bigl(F_\xi(L)\bigr)\leqslant\frac{1}{1+B},
\]
and $F_\xi$ is not a $\mathit{PDP}$-transformation.
\end{proof}

\begin{corollary}
Let $\inf_{i,j} q_{ij}:=q_{\min}$. Suppose that $q_{\min}>0$. Let
\[
\begin{aligned} & T:= \biggl\{ k: k\in\mathbb{N}, p_k<
\frac{q_{\min}}{2} \biggr\};
\\
& T_k:=T\cap\{1,2,\dots,k\};
\\
& B:=\limsup_{k\rightarrow\infty} \frac{\sum_{j\in T_k} \ln\frac{1}{p_j}}{k}. \end{aligned}
\]

Let $F_\xi$ be the distribution function of a random variable \(\xi\)
with independent $Q^*$-representation. Then $F_\xi$ preserves the
packing dimension if and only if
\[
\left\{ %
\begin{aligned} &\dim_P \mu_\xi=1;
\\
&B=0. \end{aligned} %
\right.
\]
\end{corollary}

\begin{corollary}
Let $s\in\mathbb{N}$, $s\geqslant2$;
\[
\begin{aligned} & T:= \biggl\{ k: k\in\mathbb{N}, p_k<
\frac{1}{2s} \biggr\};
\\
& T_k:=T\cap\{1,2,\dots,k\};
\\
& B:=\limsup_{k\rightarrow\infty} \frac{\sum_{j\in T_k} \ln\frac{1}{p_j}}{k}. \end{aligned}
\]

Let $F_\xi$ be the distribution function of a random variable \(\xi\)
with independent $s$-adic digits. Then $F_\xi$ preserves the packing
dimension if and only if
\[
\begin{cases}
\dim_P \mu_\xi=1;\\
B=0.
\end{cases} %
\]
\end{corollary}

%

\end{document}